\documentclass[10pt,twoside,reqno]{amsart}

\usepackage{amssymb}
\usepackage{amsmath}
\usepackage{dsfont}

\newtheorem{theorem}{\sc\hspace{8pt} Theorem}
\newtheorem{lemma}[theorem]{\sc\hspace{8pt} Lemma}
\newtheorem{corollary}[theorem]{\sc\hspace{8pt} Corollary}

\newcommand{\abs}[1]{\left\lvert#1\right\rvert}
\newcommand{\Mod}[1]{\ (\mathrm{mod}\ #1)}

\begin{document}

\title[Linnik's large sieve and the $L^{1}$ norm of exponential sums]{Linnik's large sieve and the $L^{1}$ norm of exponential sums}

\author[Emily Eckels]{Emily Eckels}

\address{Department of Mathematics, Emory University, 201 Dowman Drive, Atlanta, Georgia 30322, USA}

\email{eneckel@emory.edu}

\author[Steven Jin]{Steven Jin}

\address{Department of Mathematics, University of Maryland, College Park, 4176 Campus Drive, William E. Kirwan Hall, College Park, Maryland 20742, USA}

\email{sjin6816@umd.edu}

\author[Andrew Ledoan]{Andrew Ledoan}

\address{Department of Mathematics, University of Tennessee at Chattanooga, 415 EMCS Building (Dept. 6956), 615 McCallie Avenue, Chattanooga, Tennessee 37403, USA}

\email{andrew-ledoan@utc.edu}

\author[Brian Tobin]{Brian Tobin}

\address{Department of Mathematics, Harvard University, One Oxford Street, Cambridge, Massachusetts 02138, USA}

\email{briantobin@college.harvard.edu}



\begin{abstract}
The method of proof of Balog and Ruzsa and the large sieve of Linnik are used to investigate the behaviour of the $L^{1}$ norm of a wide class of exponential sums over the square-free integers and the primes. Further, a new proof of the lower bound due to Vaughan for the $L^{1}$ norm of an exponential sum with the von Mangoldt $\Lambda$ function over the primes is furnished. Ramanujan's sum arises naturally in the proof, which also employs Linnik's large sieve.
\end{abstract}

\maketitle

\thispagestyle{empty}

\section*{Introduction}

The $L^{1}$ norm of various exponential sums whose coefficients are taken to be arithmetical functions, such as the M\"{o}bius $\mu$ and von Mangoldt $\Lambda$ functions, as well as the characteristic function of smooth numbers, arises in many interesting problems in analytic number theory. For example, Balog and Perelli \cite{BalogPerelli1998} have proved that, for some constant $A > 0$ independent of $N$, where $N$ shall henceforth be an integer and $N \geq 2$,
\begin{equation*}
\exp \left(\frac{A \log N}{\log \log 2 N}\right)
 \ll \int_{0}^{1} \left|\sum_{n = 1}^{N} \mu (n) e (n \alpha)\right| \, d \alpha
 \ll N^{1 / 2}.
\end{equation*}
Here, $e (x) = e^{2 \pi i x}$ for $x \in \mathds{R}$.

McGehee, Pigno, and Smith \cite{McGegeePignoSmith1981} solved entirely a problem of Littlewood \cite{HardyLittlewood1948} in classical Fourier analysis concerning a lower bound for the $L^{1}$ norm of certain exponential sums; namely, that
\begin{equation*}
\int_{0}^{1} \left|\sum_{n = 1}^{N} a_{n} e (n \alpha)\right| \, d \alpha
 \gg \log N,
\end{equation*}
whenever the coefficients $a_{n}$ are arbitrary complex numbers satisfying
\begin{equation*}
\sum_{n = 1}^{N} \abs{a_{n}}^{2}
 \gg N.
\end{equation*}
Balog and Ruzsa \cite{BalogRuzsa1999} were able to show that a modest generalization to the core assumptions underlying McGehee et al.'s result, that the coefficients $a_{n}$ be zero for non-square-free integers $n$, improves the lower bound for the $L^{1}$ norm to a power of $N$. More precisely, Balog and Ruzsa introduced a simple and elegant method of proof which shows that
\begin{equation} \label{equation-1}
\int_{0}^{1} \left|\sum_{n = 1}^{N} a_{n} e (n \alpha)\right| \, d \alpha
 \gg \frac{1}{N^{3 / 8} \log N} \left(\sum_{n = 1}^{N} \abs{a_{n}}^{2}\right)^{1 / 2},
\end{equation}
whenever the coefficients $a_{n}$ are arbitrary complex numbers satisfying $a_{n} = 0$ for non-square-free integers $n$; a condition we shall henceforth understand to mean that the coefficients $a_{n}$ are supported on the square-free integers $n$.

In the case when the coefficients $a_{n}$ are the values of the M\"{o}bius $\mu$ function, the method gives a substantial improvement over the previous lower bound by Balog and Perelli, namely,
\begin{equation} \label{equation-2}
\int_{0}^{1} \left|\sum_{n = 1}^{N} \mu (n) e (n \alpha)\right| \, d \alpha
 \gg \frac{N^{1 / 8}}{\log N}.
\end{equation}
Balog and Rusza \cite{BalogRuzsa2001} later improved this lower bound to $\gg N^{1 / 6}$, using additional ideas, counting the relation \eqref{equation-17} below. However, these ideas do not improve on the lower bound \eqref{equation-1}.

The behaviour of the $L^{1}$ norm of a different type of exponential sum over the primes was investigated by Vaughan \cite{Vaughan1988}, who proved that there is a constant $B > 0$ independent of $N$ such that
\begin{equation} \label{equation-3}
\int_{0}^{1} \left|\sum_{n = 1}^{N} \Lambda (n) e (n \alpha)\right| \, d \alpha
 \geq B N^{1 / 2}.
\end{equation}
This lower bound is close to the best possible result; for it is trivial from Cauchy's inequality, Parseval's identity, and the prime number theorem that
\begin{equation*}
\begin{split}
\int_{0}^{1} \left|\sum_{n = 1}^{N} \Lambda (n) e (n \alpha)\right| \, d \alpha
 &\leq \left(\int_{0}^{1} \left|\sum_{n = 1}^{N} \Lambda (n) e (n \alpha)\right|^{2} \, d \alpha\right)^{1 / 2} \\
 &= \left(\sum_{n = 1}^{N} \Lambda (n)^{2}\right)^{1 / 2} \\
 &\sim ((1 + o (1)) N \log N)^{1 / 2}
\end{split}
\end{equation*}
as $N$ tends to infinity. Further, Vaughan felt that ``it seems quite likely'' that there is a constant $C > 0$ independent of $N$ such that
\begin{equation*}
\int_{0}^{1} \left|\sum_{n = 1}^{N} \Lambda (n) e (n \alpha)\right| \, d \alpha
 \sim C (N \log N)^{1 / 2}
\end{equation*}
as $N$ tends to infinity, ``but if true this must lie very deep.'' This problem remains unsolved.

One further result in this direction is known. A delicate study by Goldston \cite{Goldston2000} shows that, for any $\epsilon > 0$,
\begin{equation*}
\int_{0}^{1} \left|\sum_{n = 1}^{N} \Lambda (n) e (n \alpha)\right| \, d \alpha
 \leq \left(\left(\frac{1}{2} + \epsilon\right) N \log N\right)^{1 / 2}
\end{equation*}
if $N \geq N_{0} (\epsilon)$.

In this paper, we shall continue a line of investigation begun by Balog and Ruzsa and employ the authors' method and Linnik's large sieve (see \cite{Bombieri1965}, \cite{Linnik1941}, \cite{Renyi1950}, and \cite{Roth1969}), in the form of the inequality \eqref{equation-12}, to obtain a square root saving for the logarithmic factor in the lower bounds \eqref{equation-1} and \eqref{equation-2}; indeed, this was stated by Balog and Ruzsa (see \cite{BalogRuzsa1999}, page 417). In the first part of the paper, we shall prove the lower bound \eqref{equation-4}. The proof is elementary and makes use of Balog and Ruzsa's construction of exponential sums that are pointwise close to the exponential sum given by \eqref{equation-7} and supported off of the square-free integers.

\begin{theorem} \label{theorem-1}
We have
\begin{equation} \label{equation-4}
\int_{0}^{1} \left|\sum_{n = 1}^{N} b_{n} e (n \alpha)\right| \, d \alpha
 \gg \frac{1}{N^{3 / 8} (\log N)^{1 / 2}} \left(\sum_{n = 1}^{N} \abs{b_{n}}^{2}\right)^{1 / 2},
\end{equation}
whenever the coefficients $b_{n}$ are arbitrary complex numbers satisfying $b_{n} = 0$ for non-square-free integers $n$.
\end{theorem}

From this we have the lower bound \eqref{equation-5}.

\begin{corollary}
We have
\begin{equation} \label{equation-5}
\int_{0}^{1} \left|\sum_{n = 1}^{N} \mu (n) e (n \alpha)\right| \, d \alpha
 \gg \frac{N^{1 / 8}}{(\log N)^{1 / 2}}.
\end{equation}
\end{corollary}

In the second part, we shall modify Balog and Ruzsa's approach of constructing exponential sums that are pointwise close to the exponential sum given by \eqref{equation-7} but supported off of the primes. Unlike the case when the exponential sums are supported off of the square-free integers, the exponential sum $H_{N, P}$ given by \eqref{equation-22}, thus obtained in the manner of Balog and Ruzsa, is only a power of the logarithm weaker than the best possible result due to Vaughan. As will be explained, it nearly achieves the lower bound \eqref{equation-3}. When the exponential sum that is supported off of the primes has a weight which oscillates, so that the exponential sum itself oscillates, there are no significant contributions at $\alpha = 0$ or at fractions, and hence Vaughan's method no longer applies. In spite of this difficulty, we are able to obtain a lower bound the size of about $N^{1 / 4}$, using Balog and Ruzsa's method.

In the third part, inspired by Vaughan's investigation, we shall give a new proof of the lower bound \eqref{equation-3}, which also employs Linnik’s large sieve. Ramanujan’s sum, given by \eqref{equation-31}, arises naturally in the proof.

\bigskip

\section*{Part I}

We consider an exponential sum
\begin{equation} \label{equation-6}
F_{N} (\alpha)
 = \sum_{n = 1}^{N} e (n \alpha)
\end{equation}
and derive its Fej\'{e}r kernel
\begin{equation} \label{equation-7}
T_{N} (\alpha)
 = \frac{1}{N} \abs{F_{N} (\alpha)}^{2}
 = \sum_{\abs{k} \leq N} \left(1 - \frac{\abs{k}}{N}\right) e (k \alpha).
\end{equation}
By \eqref{equation-7} we have
\begin{equation*}
\begin{split}
\sum_{a = 1}^{q} T_{N} \left(\alpha - \frac{a}{q}\right)
 &= q \sum_{\substack{\abs{k} \leq N \\ q \mid k}} \left(1 - \frac{\abs{k}}{N}\right) e (k \alpha) \\
 &= \sum_{\abs{k} \leq N} \left(1 - \frac{\abs{k}}{N}\right) \epsilon_{q} (k) e (k \alpha),
\end{split}
\end{equation*}
where
\begin{equation*}
\epsilon_{q} (n)
 = \sum_{a = 1}^{q} e \left(-\frac{n a}{q}\right)
 = \left\{ \begin{array}{ll}
   q, & \mbox{if $q \mid n$,} \\
   0,   & \mbox{if $q \nmid n$.}
\end{array}
\right.
\end{equation*}
We define
\begin{align} \label{equation-8}
G^{\ast}_{N} (\alpha)
 &= \frac{1}{\pi (P)} \sum_{p \leq P} \sum_{a = 1}^{p^{2}} T_{N} \left(\alpha - \frac{a}{p^{2}}\right) \nonumber \\
 &= T_{N} (\alpha) + \frac{1}{\pi (P)} \sum_{p \leq P} \sum_{a = 1}^{p^{2} - 1} T_{N} \left(\alpha - \frac{a}{p^{2}}\right),
\end{align}
where $p$ denotes a prime and $\pi (x)$ stands for the number of primes up to $x$. We have also the alternative form
\begin{equation*}
\begin{split}
G^{\ast}_{N} (\alpha)
 &=\frac{1}{\pi (P)} \sum_{p \leq P} \sum_{\abs{k} \leq N} \left(1 - \frac{\abs{k}}{N}\right) \epsilon_{p^{2}} (k) e (k \alpha) \\
 &= \sum_{\abs{k} \leq N} \left(1 - \frac{\abs{k}}{N}\right) c_{k} e (k \alpha),
\end{split}
\end{equation*}
where
\begin{equation*}
c_{k}
 = \frac{1}{\pi (P)} \sum_{p \leq P} \epsilon_{p^{2}} (k).
\end{equation*}
It follows that we must have $c_{k} = 0$ if $p^{2} \nmid k$ for all primes $p$ with $p \leq P$. Plainly, $c_{k} \neq 0$ if and only if there exists a prime $p$ with $p \leq P$ such that $p^{2} \mid k$. The values of $k$ where this is true are the non-square-free integers with a square factor $p^{2}$ with $p \leq P$. Thus, $c_{k} = 0$ on the square-free integers. We wish to prove that $G^{\ast}_{N}$ gives a very good approximation of $T_{N}$ in the following sense.

\begin{lemma}
We have
\begin{equation} \label{equation-9}
\abs{G^{\ast}_{N} (\alpha) - T_{N} (\alpha)}
 \ll N^{3 / 4} \log N
\end{equation}
uniformly in $\alpha \in \mathds{R}$.
\end{lemma}

\begin{proof}
It is well known that
\begin{equation} \label{equation-10}
\abs{T_{N} (\alpha)}
 \ll \min \left(N, \frac{1}{N \|\alpha\|^{2}}\right),
\end{equation}
where $\|x\|$ is the distance from $x$ to the nearest integer, that is,
\begin{equation*}
\|x\|
 = \inf_{n \in \mathds{Z}} \abs{x - n}.
\end{equation*}
Therefore, we obtain from \eqref{equation-8} and \eqref{equation-10} that
\begin{equation*}
\begin{split}
0
 \leq G^{\ast}_{N} (\alpha) - T_{N} (\alpha)
 &= \frac{1}{\pi (P)} \sum_{p \leq P} \sum_{a = 1}^{p^{2} - 1} T_{N} \left(\alpha - \frac{a}{p^{2}}\right) \\
 &\ll \frac{1}{\pi (P)} \sum_{p \leq P} \min \left(N, \frac{1}{N \|\alpha - a / p^{2}\|^{2}}\right).
\end{split}
\end{equation*}
Then if $P = N^{1 / 4}$, for a fixed $\alpha$, the shifted fractions $\alpha - a / p^{2}$ are all distinct and well spaced by at least $1 / N$. The fraction closest to $\alpha$ is estimated by $N$. Further, the $k$th fraction is at least $k / (2 N)$ apart from $\alpha$. Hence,
\begin{equation*}
\begin{split}
0
 \leq G^{\ast}_{N} (\alpha) - T_{N} (\alpha)
 &\ll \frac{1}{\pi (P)} N \sum_{k = 1}^{N} \frac{1}{k^{2}} \\
 &\ll \frac{1}{\pi (P)} N \\
 &\ll N^{3 / 4} \log N,
\end{split}
\end{equation*}
as required.
\end{proof}

We shall give an alternative proof of the upper bound \eqref{equation-9}, which avoids these calculations. From \eqref{equation-7} and \eqref{equation-8} it follows at once that
\begin{equation} \label{equation-11}
0
 \leq G^{\ast}_{N} (\alpha) - T_{N} (\alpha)
 = \frac{1}{\pi (P)} \sum_{p \leq P} \sum_{a = 1}^{p^{2} - 1} \frac{1}{N} \left|F_{N} \left(\alpha - \frac{a}{p^{2}}\right)\right|^{2}.
\end{equation}
We consider the following optimal version of the large sieve: Let $\alpha_{r}$, where $1 \leq r \leq R$ and $R \geq 2$, be distinct points modulo one, and let $\delta > 0$ be such that
\begin{equation*}
\|\alpha_{r} - \alpha_{s}\|
 \geq \delta
\end{equation*}
for $r \neq s$. Then for arbitrary complex numbers $a_{n}$,
\begin{equation} \label{equation-12}
\sum_{r = 1}^{R} \left|\sum_{n = M + 1}^{M + N} a_{n} e (n \alpha_{r})\right|^{2}
 \leq \left(N + \frac{1}{\delta} - 1\right) \sum_{n = M + 1}^{M + N} \abs{a_{n}}^{2},
\end{equation}
where $M$ and $N$ are integers and $N > 0$.

The stated constant $N + 1 / \delta - 1$ is sharp and was contributed by Selberg (see Chapter 27 in \cite{Davenport2000}). We can apply the large sieve inequality \eqref{equation-12} with $\delta = 1 / P^{4}$ to obtain
\begin{equation} \label{equation-13}
\sum_{p \leq P} \sum_{a = 1}^{p^{2} - 1} \left|F_{N} \left(\alpha - \frac{a}{p^{2}}\right)\right|^{2}
 \leq (N + P^{4} -1) N.
\end{equation}
Then, from \eqref{equation-11} and \eqref{equation-13}, we have, for sufficiently large $N$,
\begin{equation*}
\begin{split}
0
 \leq G^{\ast}_{N} (\alpha) - T_{N} (\alpha)
 &\leq \frac{1}{\pi (P)} (N + P^{4} - 1) \\
 &\leq \frac{\log N^{1 / 4}}{N^{1 / 4}} (2 N - 1) \\
 &\leq  \frac{1}{2} N^{3 / 4} \log N,
\end{split}
\end{equation*}
if, as was supposed, $P = N^{1 / 4}$, with the inequality
\begin{equation} \label{equation-14}
\pi (N) > \frac{N}{\log N},
\end{equation}
which holds for $N \geq 17$ (see Corollary 1, Inequality (3.5), in \cite{RosserSchoenfeld1962}), applied in the penultimate step.

We now proceed to the proof of Theorem \ref{theorem-1}.

\begin{proof}[Proof of Theorem \ref{theorem-1}]
For $N \geq 2$ and arbitrary complex numbers $a_{n}$ satisfying $a_{n} = 0$ for non-square-free integers $n$, let there be
\begin{equation*}
M_{N} (\alpha)
 = \sum_{n = 1}^{N} a_{n} e (n \alpha).
\end{equation*}
We have, for any fixed $\alpha$,
\begin{equation*}
\begin{split}
\int_{0}^{1} G^{\ast}_{N} (\alpha - \beta) M_{N} (\beta) \, d \beta
 &= \sum_{\abs{k} \leq N} \left(1 - \frac{\abs{k}}{N}\right) c_{k} e (k \alpha) \sum_{n = 1}^{N} a_{n} \int_{0}^{1} e ((n - k) \beta) \, d \beta \\
 &= \sum_{n = 1}^{N} \left(1 - \frac{n}{N}\right) a_{n} c_{n} e (n \alpha)
 = 0,
\end{split}
\end{equation*}
since the $a_{n}$ are supported on the square-free integers, whereas the $c_{n}$ are supported off of the square-free integers.
In like manner, we find
\begin{equation*}
\int_{0}^{1} T_{N} (\alpha - \beta) M_{N} (\beta) \, d \beta
 = \sum_{n = 1}^{N} \left(1 - \frac{n}{N}\right) a_{n} e (n \alpha).
\end{equation*}
We have therefore
\begin{equation*}
\sum_{n = 1}^{N} \left(1 - \frac{n}{N}\right) a_{n} e (n \alpha)
 = \int_{0}^{1} (T_{N} (\alpha - \beta) - G^{\ast}_{N} (\alpha - \beta)) M_{N} (\beta) \, d \beta.
\end{equation*}
Then \eqref{equation-9} implies that
\begin{align}
\left|\sum_{n = 1}^{N} \left(1 - \frac{n}{N}\right) a_{n} e (n \alpha)\right|
 &\leq \int_{0}^{1} \abs{T_{N} (\alpha - \beta) - G^{\ast}_{N} (\alpha - \beta)} \abs{M_{N} (\beta)} \, d \beta \label{equation-15} \\
 &\ll N^{3 / 4} \log N \int_{0}^{1} \abs{M_{N} (\beta)} \, d \beta. \label{equation-16}
\end{align}

We now introduce an exponential sum
\begin{equation*}
g_{N} (\alpha)
 = \sum_{n = 1}^{N} b_{n} e (n \alpha),
\end{equation*}
where the coefficients $b_{n}$ are arbitrary complex numbers satisfying $b_{n} = 0$ for non-square-free integers $n$. We have
\begin{equation*}
\int_{0}^{1} \overline{g_{N} (\beta)} g_{N} (\alpha + \beta) \, d \beta
 = \sum_{n = 1}^{N} \abs{b_{n}}^{2} e (n \alpha).
\end{equation*}
It thus follows that
\begin{align} \label{equation-17}
\int_{0}^{1} \left|\sum_{n = 1}^{N} \abs{b_{n}}^{2} e (n \alpha)\right| \, d \alpha
 &\leq \int_{0}^{1} \int_{0}^{1} \abs{g_{N} (\beta)} \abs{g_{N} (\alpha + \beta)} \, d \beta \, d \alpha \nonumber \\
 &= \left(\int_{0}^{1} \abs{g_{N} (\alpha)} \, d \alpha\right)^{2},
\end{align}
and thence we find, on setting
\begin{equation*}
M_{N} (\alpha)
 = \sum_{n = 1}^{N} \abs{b_{n}}^{2} e (n \alpha),
\end{equation*}
from \eqref{equation-16} and \eqref{equation-17} that
\begin{equation*}
\begin{split}
\left|\sum_{n = 1}^{N} \left(1 - \frac{n}{N}\right) \abs{b_{n}}^{2} e (n \alpha)\right|
 &\ll N^{3 / 4} \log N \int_{0}^{1} \left|\sum_{n = 1}^{N} \abs{b_{n}}^{2} e (n \alpha)\right| \, d \alpha \\
 &\leq N^{3 / 4} \log N \left(\int_{0}^{1} \left|\sum_{n = 1}^{N} b_{n} e (n \alpha)\right| \, d \alpha\right)^{2}.
\end{split}
\end{equation*}
A particular case of the data is that, if we set $\alpha = 0$ on the extreme left side, then
\begin{equation*}
\sum_{n = 1}^{N} \left(1 - \frac{n}{N}\right) \abs{b_{n}}^{2}
 \ll N^{3 / 4} \log N \left(\int_{0}^{1} \left|\sum_{n = 1}^{N} b_{n} e (n \alpha)\right| \, d \alpha\right)^{2}.
\end{equation*}
Thus, by taking $b_{n} = 0$ for $ M = N / 2 < n \leq N$ and assuming that $N$ is even, we can make
\begin{equation*}
\begin{split}
\sum_{n = 1}^{M} \abs{b_{n}}^{2}
 &\leq 2 \sum_{n = 1}^{M} \left(1 - \frac{n}{2 M}\right) \abs{b_{n}}^{2} \\
 &\ll M^{3 / 4} \log M \left(\int_{0}^{1} \left|\sum_{n = 1}^{M} b_{n} e (n \alpha)\right| \, d \alpha\right)^{2}.
\end{split}
\end{equation*}
Hence, the stated result is entirely proved.
\end{proof}

We remark that Balog and Ruzsa's method actually establishes the stated result from \eqref{equation-15}, \eqref{equation-16}, and Parseval's identity.

\bigskip

\section*{Part II}

We first endeavor to construct an exponential sum that is supported off of the primes that exceed $N^{1 / 2}$. We modify the initial proof in Part I by considering
\begin{align} \label{equation-18}
H_{N} (\alpha)
 &= \frac{1}{\pi (P)} \sum_{p \leq P} \sum_{\abs{k} \leq N} \left(1 - \frac{\abs{k}}{N}\right) \epsilon_{p} (k) e (k \alpha) \nonumber \\
 &= T_{N} (\alpha) + \frac{1}{\pi (P)} \sum_{p \leq P} \sum_{a = 1}^{p - 1} T_{N} \left(\alpha - \frac{a}{p}\right).
\end{align}
It will be convenient to define
\begin{align} \label{equation-19}
H_{N} (\alpha)
 &= \frac{1}{\pi (P)} \sum_{p \leq P} \sum_{\abs{k} \leq N} \left(1 - \frac{\abs{k}}{N}\right) \epsilon_{p} (k) e (k \alpha) \nonumber \\
 &= \sum_{\abs{k} \leq N} \left(1 - \frac{\abs{k}}{N}\right) d_{k} e (k \alpha),
\end{align}
where
\begin{equation} \label{equation-20}
d_{k}
 = \frac{1}{\pi (P)} \sum_{p \leq P} \epsilon_{p} (k).
\end{equation}
We have now to observe that $d_{k} \neq 0$ if and only if there exists a prime $p$ with $p \leq P$ such that $p \mid k$. Clearly, $d_{k} = 0$ if and only if the smallest prime factor of $k$ is greater than $P$ or $k = 1$. We will take $P = N^{1 / 2}$, so that $H_{N}$ is supported off of the primes $p$ in the range $P < p \leq N$. We prove that $H_{N}$ is a very good approximation of $T_{N}$.

\begin{lemma}
We have
\begin{equation} \label{equation-21}
\abs{H_{N} (\alpha) - T_{N} (\alpha)}
 \ll N^{1 / 2} \log N
\end{equation}
uniformly in $\alpha \in \mathds{R}$.
\end{lemma}

\begin{proof}
We start with \eqref{equation-7} and \eqref{equation-18} and compute, by means of the large sieve inequality \eqref{equation-12} with $\delta = 1 / P^{2}$, for sufficiently large $N$,
\begin{equation*}
\begin{split}
\abs{H_{N} (\alpha) - T_{N} (\alpha)}
 &= \frac{1}{\pi (P)} \left|\sum_{p \leq P} \sum_{a = 1}^{p - 1} T_{N} \left(\alpha - \frac{a}{p}\right)\right| \\
 &= \frac{1}{\pi (P)} \sum_{p \leq P} \sum_{a = 1}^{p - 1} \frac{1}{N} \left|F_{N} \left(\alpha - \frac{a}{p}\right)\right|^{2} \\
 &\leq \frac{\log P}{P} (N + P^{2} - 1) \\
 &\leq N^{1 / 2} \log N,
\end{split}
\end{equation*}
if $P = N^{1 / 2}$, again by \eqref{equation-14} in the second to last step.
\end{proof}

Having thus defined $H_{N}$, we shall now define
\begin{align} \label{equation-22}
H_{N, P} (\alpha)
 &= \frac{1}{\pi (P)} \sum_{p \leq P} \sum_{P < \abs{k} \leq N} \left(1 - \frac{\abs{k}}{N}\right) \epsilon_{p} (k) e (k \alpha) \nonumber \\
 &= \sum_{P < \abs{k} \leq N} \left(1 - \frac{\abs{k}}{N}\right) d_{k} e (k \alpha).
\end{align}
Clearly, $H_{N, P}$ is supported off of all the primes. From \eqref{equation-19}, \eqref{equation-20}, and \eqref{equation-22}, we have
\begin{equation*}
\begin{split}
\abs{H_{N} (\alpha) - H_{N, P} (\alpha)}
 &= \left|\sum_{\abs{k} \leq P} \left(1 - \frac{\abs{k}}{N}\right) d_{k} e (k \alpha)\right| \\
 &\leq \sum_{\abs{k} \leq P} \abs{d_{k}}
 = \frac{1}{\pi (P)} \sum_{\abs{k} \leq P} \sum_{p \leq P} \abs{ \epsilon_{p} (k)} \\
 &= \frac{1}{\pi (P)} \left(\sum_{p \leq P} p + \sum_{1 \leq \abs{k} \leq P} \sum_{\substack{p \leq P \\ p \mid k}} p\right) \\
 &= \frac{1}{\pi (P)} \left(\sum_{p \leq P} p + \sum_{1 \leq \abs{p m} \leq P} p\right)\\
 &= \frac{1}{\pi (P)} \left( \sum_{p \leq P} p +  \sum_{p \leq P} p \sum_{1 \leq \abs{m} \leq P / p} 1\right) \\
 &\leq \frac{1}{\pi (P)} (P \pi (P) + 2 P \pi (P))
 = 3 N^{1 / 2},
\end{split}
\end{equation*}
if $P = N^{1 / 2}$. We have, then, produced from this estimate and \eqref{equation-21} the following crucial lemma.

\begin{lemma}
We have, with $P = N^{1 / 2}$,
\begin{equation} \label{equation-23}
\abs{H_{N, P} (\alpha) - T_{N} (\alpha)}
 \ll N^{1 / 2} \log N
\end{equation}
uniformly in $\alpha \in \mathds{R}$.
\end{lemma}

In virtue of the results established in Part I, it is natural to study the $L^{1}$ norm of the exponential sum
\begin{equation} \label{equation-24}
S_{N} (\alpha)
 = \sum_{p \leq N} a_{p} e (p \alpha)
 = \sum_{n = 1}^{N} \mathbf{1}_{p} (n) a_{n} e (n \alpha),
\end{equation}
where the $a_{n}$ are now arbitrary complex numbers and $\mathbf{1}_{p}$ is the indicator function of the primes, that is,
\begin{equation*}
\mathbf{1}_{p} (n)
 = \left\{ \begin{array}{ll}
   1, & \mbox{if $n = p$,} \\
   0,   & \mbox{otherwise.}
\end{array}
\right.
\end{equation*}
There are other exponential sums to consider, such as $\sum_{n = 1}^{N} \theta (n) a_{n} e (n \alpha)$, where $\theta (n) = \log p$ if $n = p$, and $\theta (n) = 0$ otherwise, and $\sum_{n = 1}^{N} \Lambda (n) a_{n} e (n \alpha)$. It is straightforward to modify a proof for $S_{N}$ to apply for these exponential sums. 

Proceeding with the same analysis from Part I, we write
\begin{equation*}
\begin{split}
\int_{0}^{1} H_{N, P} (\alpha - \beta) S_{N} (\beta) \, d \beta
 &= \sum_{n = 1}^{N} \mathbf{1}_{p} (n) a_{n} \sum_{P < \abs{k} \leq N} \left(1 - \frac{\abs{k}}{N}\right) d_{k} e (k \alpha)  \\ &\hspace{95pt} \times \int_{0}^{1} e ((n - k) \beta) \, d \beta \\
 &= \sum_{P < n \leq N} \left(1 - \frac{n}{N}\right) \mathbf{1}_{p} (n) a_{n} d_{n} e (n \alpha)
 = 0
\end{split}
\end{equation*}
and observe that
\begin{equation*}
\int_{0}^{1} T_{N} (\alpha - \beta) S_{N} (\beta) \, d \beta
 = \sum_{n = 1}^{N} \left(1 - \frac{n}{N}\right) \mathbf{1}_{p} (n) a_{n} e (n \alpha).
\end{equation*}
Thus, we have
\begin{equation} \label{equation-25}
\begin{split}
\sum_{n = 1}^{N} \left(1 - \frac{n}{N}\right) \mathbf{1}_{p} (n) a_{n} e (n \alpha)
 &= \int_{0}^{1} (T_{N} (\alpha - \beta) - H_{N, P} (\alpha - \beta)) S_{N} (\beta) \, d \beta.
\end{split}
\end{equation}
Hence, applying \eqref{equation-23} and relabelling we obtain the following result.

\begin{theorem}
We have
\begin{equation} \label{equation-26}
\left|\sum_{n = 1}^{N} \left(1 - \frac{n}{N}\right) \mathbf{1}_{p} (n) a_{n} e (n \alpha)\right|
 \ll N^{1 / 2} \log N \int_{0}^{1} \left|S_{N} (\alpha)\right| \, d \alpha
\end{equation}
uniformly in $\alpha$, where $S_{N}$ is the exponential sum given by \eqref{equation-24}, the $a_{n}$ are arbitrary complex numbers, and $\mathbf{1}_{p}$ is the indicator function of the primes.
\end{theorem}

We illustrate \eqref{equation-26} by an application. If we put $a_{n} = 1$ and $\alpha = 0$, then the left side of \eqref{equation-25} reduces to
\begin{equation*}
\begin{split}
\sum_{n = 1}^{N} \left(1 - \frac{n}{N}\right) \mathbf{1}_{p} (n)
 &= \sum_{p \leq N} \left(1 - \frac{p}{N}\right) \\
 &\sim \int_{2}^{N} \left(1 - \frac{u}{N}\right) \frac{\, d u}{\log u} \\
 &\sim \frac{N}{2 \log N}
\end{split}
\end{equation*}
as $N$ tends to infinity, by means of the prime number theorem. Then from this estimate and \eqref{equation-26} we obtain the following result.

\begin{theorem}
We have
\begin{equation} \label{equation-27}
\int_{0}^{1} \left|\sum_{p \leq N} e (p \alpha)\right| \, d \alpha
 \gg \frac{N^{1 / 2}}{(\log N)^{2}}.
\end{equation}
\end{theorem}

Vaughan \cite{Vaughan1988} has proved that the lower bound for the $L^{1}$ norm in \eqref{equation-27} is $\gg N^{1 / 2} / \log N$. Vaughan's proof only seems to work for $S_{N}$ with the $a_{n}$ being a smooth continuous function. Indeed, Vaughan's lower bound depends on the sum $\sum_{n = 1}^{N} \mathbf{1}_{p} (n) a_{n}$ not cancelling out. Exactly as in \cite{BalogRuzsa1999}, we shall now prove the corresponding result for $S_{N}$.

\begin{theorem}
We have
\begin{equation} \label{equation-28}
\int_{0}^{1} \left|S_{N} (\alpha)\right| \, d \alpha
 \gg \frac{1}{N^{1 / 4} (\log N)^{1 / 2}} \left(\sum_{n = 1}^{N} \mathbf{1}_{p} (n) \abs{a_{n}}^{2}\right)^{1 / 2},
\end{equation}
where $S_{N}$ is the exponential sum given by \eqref{equation-24}, the $a_{n}$ are arbitrary complex numbers, and $\mathbf{1}_{p}$ is the indicator function of the primes.
\end{theorem}

\begin{proof}
We follow the same proof given at the end of Part I. From \eqref{equation-17}, for arbitrary complex numbers $b_{n}$, we have
\begin{equation*}
\int_{0}^{1} \left|\sum_{n = 1}^{N} b_{n} e (n \alpha)\right| \, d \alpha
 \geq \left(\int_{0}^{1} \left|\sum_{n = 1}^{N} \abs{b_{n}}^{2} e (n \alpha)\right| \, d \alpha\right)^{1 / 2}.
\end{equation*}
We shall suppose that $N$ is an even integer and define
\begin{equation*}
b_{n}
 = \left\{ \begin{array}{ll}
   \mathbf{1}_{p} (n) a_{n}, & \mbox{if $n \leq N / 2$,} \\
   0,   & \mbox{if $N / 2 < n \leq N$.}
\end{array}
\right.
\end{equation*}
Thus, taking $M = N / 2$, we see that
\begin{equation*}
\begin{split}
\int_{0}^{1} \left|S_{M} (\alpha)\right| \, d \alpha
 &\geq \left(\int_{0}^{1} \left|\sum_{n = 1}^{M} \mathbf{1}_{p} (n) \abs{a_{n}}^{2} e (n \alpha)\right| \, d \alpha\right)^{1 / 2} \\
 &\gg \frac{1}{M^{1 / 4} (\log M)^{1 / 2}} \left|\sum_{n = 1}^{M} \left(1 - \frac{n}{2 M}\right) \mathbf{1}_{p} (n) \abs{a_{n}}^{2} e (n \alpha)\right|^{1 / 2},
\end{split}
\end{equation*}
by virtue of \eqref{equation-26}. Choosing $\alpha = 0$ on the extreme right side, we decrease this lower bound by removing the factor $1 - n / (2 M)$, and relabelling gives the required result.
\end{proof}

As an application of \eqref{equation-28}, when the $a_{n}$ are taken to be $\chi_{3}$, the non-principal Dirichlet character modulo three defined by
\begin{equation*}
\chi_{3} (n)
 = \left\{ \begin{array}{ll}
   0, & \mbox{if $n \equiv 0 \Mod{3}$,} \\
   1, & \mbox{if $n \equiv 1 \Mod{3}$,} \\
   -1, & \mbox{if $n \equiv 2 \Mod{3}$,}
\end{array}
\right.
\end{equation*}
we have
\begin{equation*}
\begin{split}
S_{N} (\alpha)
 &= \sum_{n = 1}^{N} \mathbf{1}_{p} (n) \chi_{3} (n) e (n \alpha) \\
 &= \sum_{\substack{p \leq N \\ p \equiv 1 \Mod{3}}} e (p \alpha) - \sum_{\substack{p \leq N \\ p \equiv 2 \Mod{3}}} e (p \alpha).
\end{split}
\end{equation*}
We take the special case when $\alpha = 0$, so that
\begin{equation*}
S_{N} (0)
 = \sum_{\substack{p \leq N \\ p \equiv 1 \Mod{3}}} 1 - \sum_{\substack{p \leq N \\ p \equiv 2 \Mod{3}}} 1,
\end{equation*}
which has similar oscillations to $\pm N^{1 / 2} / \log N$, and occasionally oscillates at least as large as $\pm (N^{1 / 2} / \log N) \log \log \log N$. Although Vaughan's method presumably fails to yield the desired lower bound at this point, we have from \eqref{equation-28} that
\begin{equation*}
\begin{split}
\int_{0}^{1} \left|\sum_{n = 1}^{N} \mathbf{1}_{p} (n) \chi_{3} (n) e (n \alpha)\right| \, d \alpha
 &\gg \frac{\pi (N)^{1 / 2}}{N^{1 / 4} (\log N)^{1 / 2}} \\
 &\gg \frac{N^{1 / 4}}{\log N}.
\end{split}
\end{equation*}
We thus obtain the following result.

\begin{theorem}
We have
\begin{equation*}
\int_{0}^{1} \left|\sum_{p \leq N} \chi_{3} (p) e (p \alpha)\right| \, d \alpha
 \gg \frac{N^{1 / 4}}{\log N},
\end{equation*}
where $\chi_{3}$ is the non-principal Dirichlet character modulo three.
\end{theorem}

Of course, the same result applies for any non-principal Dirichlet character, while the stronger result \eqref{equation-27} will hold for principal Dirichlet characters.

\bigskip

\section*{Part III}

Our proof of Vaughan's lower bound \eqref{equation-3} rests upon the following lemma.

\begin{lemma}
Let
\begin{equation*}
V
 = \int_{0}^{1} \sum_{n = 1}^{N} \Lambda (n) e (n \alpha) K_{N, Q} (\alpha) \, d \alpha,
\end{equation*}
where
\begin{equation*}
K_{N, Q} (\alpha)
 = \sum_{q \leq Q} \mu (q) \sum_{\substack{a = 1 \\ (q, a) = 1}}^{q} \left|F_{N} \left(\alpha - \frac{a}{q}\right)\right|^{2}
\end{equation*}
and $F_{N}$ is the exponential sum given by \eqref{equation-6}. Suppose that $Q = f (N)$ tends to infinity with $N$ and that $f (N) \leq o (N)$. Then we have
\begin{equation} \label{equation-29}
V
 \sim \frac{3 Q}{\pi^{2}} N^{2}.
\end{equation}
\end{lemma}

\begin{proof}
Since, obviously, by \eqref{equation-7}
\begin{equation*}
\abs{F_{N} (\alpha)}^{2}
 = \sum_{\abs{k} \leq N} (N - \abs{k}) e (k \alpha),
\end{equation*}
then
\begin{equation} \label{equation-30}
V
 = \sum_{q \leq Q} \mu (q) \sum_{n = 1}^{N} (N - n) \Lambda (n) c_{q} (-n),
\end{equation}
where $c_{q}$ is Ramanujan's sum (see \cite{Ramanujan1918}) defined by
\begin{equation} \label{equation-31}
c_{q} (n)
 = \sum_{\substack{a = 1 \\ (q, a) = 1}}^{q} e \left(\frac{a n}{q}\right).
\end{equation}
Now $c_{q} (-n) = c_{q} (n)$, and $c_{q} (n) = \mu (q)$ if $(q, n) = 1$. Further, it is trivially true that $\abs{c_{q} (n)} \leq \phi (q) \leq q$, where $\phi$ is Euler's totient function. Thus, we have
\begin{equation*}
\sum_{q \leq Q} \mu (q) c_{q} (n)
 = \sum_{q \leq Q} \mu (q)^{2} + \sum_{\substack{q \leq Q \\ (q, n) > 1}} \mu (q) (c_{q} (n) - \mu (q)).
\end{equation*}
Since the number of square-free integers not exceeding $Q$ has asymptotic density
\begin{equation*}
\sum_{q \leq Q} \mu (q)^{2}
 = \frac{6}{\pi^{2}} Q + O (Q^{1 / 2})
\end{equation*}
as $Q$ tends to infinity (see Theorem 2.2 in \cite{MontgomeryVaughan2007}), it follows that
\begin{equation*}
\sum_{q \leq Q} \mu (q) c_{q} (n)
 = \frac{6}{\pi^{2}} Q + O (Q^{1 / 2}) + O \left(\sum_{\substack{q \leq Q \\ (q, n) > 1}} \mu (q)^{2} q\right).
\end{equation*}
Inserting this into \eqref{equation-30}, we obtain
\begin{equation} \label{equation-32}
\begin{split}
V
 &= \left(\frac{6}{\pi^{2}} Q + O (Q^{1 / 2})\right) \sum_{n = 1}^{N} (N - n) \Lambda (n) \\ &\hspace{60pt} + O \left(\sum_{q \leq Q} \mu (q)^{2} q \sum_{\substack{n = 1 \\ (q, n) > 1}}^{N} (N - n) \Lambda (n)\right).
\end{split}
\end{equation}
In the second error term on the right side of \eqref{equation-32} we see that $n = p^{m}$ and $q$ is square-free. Thus, the condition $(q, n) > 1$ implies that $q = p$. Hence, this error term is at most
\begin{align} \label{equation-33}
O \left(N \sum_{p \leq Q} p \log p \sum_{m \leq \log N / \log p} 1\right)
 &= O \left (N \log N \sum_{p \leq Q} p\right) \nonumber \\
 &= O \left(\frac{Q^{2}}{\log Q} N \log N\right),
\end{align}
by the prime number theorem, which also implies that
\begin{equation} \label{equation-34}
\sum_{n = 1}^{N} (N - n) \Lambda (n)
 = \frac{1}{2} N^{2} (1 + o (1)).
\end{equation}
Combining \eqref{equation-33} and \eqref{equation-34} in \eqref{equation-32}, we obtain
\begin{equation*}
V
 = \frac{3 Q}{\pi^{2}} N^{2} (1 + o (1)) + O \left(\frac{Q^{2}}{\log Q} N \log N\right),
\end{equation*}
and hence the result follows.
\end{proof}

For the proof of the lower bound \eqref{equation-3}, we observe that
\begin{equation} \label{equation-35}
V
 \leq \max_{0 \leq \alpha \leq 1} \abs{K_{N, Q} (\alpha)} \int_{0}^{1} \left|\sum_{n = 1}^{N} \Lambda (n) e (n \alpha)\right|\, d \alpha.
\end{equation}
We apply the large sieve inequality \eqref{equation-12} with $\delta = 1 / Q^{2}$ to obtain
\begin{align} \label{equation-36}
\abs{K_{N, Q} (\alpha)}
 &\leq \sum_{q \leq Q} \sum_{\substack{a = 1 \\ (q, a) = 1}}^{q} \left|F_{N} \left(\alpha - \frac{a}{q}\right)\right|^{2} \nonumber \\
 &\leq (N + Q^{2} - 1) \sum_{n = 1}^{N} 1 \nonumber \\
 &< N (N + Q^{2}).
\end{align}
By virtue of \eqref{equation-29}, \eqref{equation-35}, and \eqref{equation-36}
\begin{equation*}
\begin{split}
\int_{0}^{1} \left|\sum_{n = 1}^{N} \Lambda (n) e (n \alpha)\right|\, d \alpha
 &\geq \frac{V}{N (N + Q^{2})} \\
 &\geq \left(\frac{3}{\pi^{2}} - \epsilon\right) \frac{Q N}{N + Q^{2}}.
\end{split}
\end{equation*}
Choosing $Q = N^{1 / 2}$, the desired result follows.

\bigskip

\section*{Acknowledgements}

Our grateful thanks are due to Prof. Daniel Alan Goldston of San Jos\'{e} State University, for drawing this interesting problem to our attention and for many valuable suggestions and advice. This research has received funding from the National Science Foundation under Grant DMS-1852288. Andrew Ledoan was partially supported by the Center of Excellence in Applied Computational Science and Engineering, under Grant FY2019 CEACSE Awards.

\bigskip

\bibliographystyle{amsplain}

\end{document}